\documentclass[reqno]{amsart}
\usepackage{amssymb,latexsym}
\usepackage{amsfonts,mathrsfs}
\usepackage[margin=1.4in]{geometry}
\usepackage{amsmath}
\usepackage{enumerate}
\usepackage{mathtools}

\newtheorem{thm}{Theorem}[section]
\newtheorem{lemma}[thm]{Lemma}

\theoremstyle{remark}
\newtheorem{remark}{Remark}[section]

\begin{document}
\bibliographystyle{abbrv}

\title{The isoperimetric inequality in Steady Ricci Solitons}
\author{Yuqiao Li}
\address{University of Science and Technology of China, No. 96, JinZhai Road Baohe District, Hefei, Anhui, 230026, P.R.China.} \email{lyq112@mail.ustc.edu.cn}
\thanks {2010 Mathematics Subject Classification 53C44}
\thanks{The research is supported by the National Nature Science Foundation of China No. 11721101 No. 11526212}

\begin{abstract}
We prove that the isoperimetric inequality is satisfied in the cigar steady soliton and in the Bryant steady soliton. Since both of them are Riemannian manifolds with warped product metric, we utilize the result of Guan-Li-Wang to get our conclusion. For the sake of the soliton structure, we believe that the geometric restrictions for manifolds in which the isoperimetric inequality holds are naturally satisfied for steady Ricci solitons.
\end{abstract}

\maketitle

\numberwithin{equation}{section}
\section{Introduction and statement of the main results}

Let $\Omega$ be a bounded domain in the two-dimensional Euclidean space. We know the isoperimetric inequality
\[ L^2\geq 4\pi A \]
holds, where $L$ and $A$ are the boundary length and the area of $\Omega$ respectively. Equality is attained only when $\Omega$ is a ball.

In $\mathbb{R}^n$, if $\Omega\subset\mathbb{R}^n$ is a bounded domain with boundary $\partial\Omega$ which is a $n-1$-dimensional hypersurface, then the corresponding  isoperimetric inequality reads
\[ Area(\partial\Omega)\geq c(n) (Vol(\Omega))^{\frac{n-1}{n}} \]
where $Area(\partial\Omega)$ denotes the $(n-1)$-Haussdorff measure of the hypersurface $\partial\Omega$ and $Vol(\Omega)$ is the $n$-dimensional volume of $\Omega$ and $c(n)$ is a constant depends only on the dimension $n$. Equality holds only when $\Omega$ is a ball.

Ricci soliton is a self-similar solution of the Ricci flow. It is obtained by a family of diffeomorphisms of the initial metric and satisfies the soliton equation
\[ -2Ric(g)=\mathcal{L}_Xg+\epsilon g \]
where $X$ is a vector field and $\epsilon$ is a constant. A Ricci soliton is called a steady soliton if $\epsilon=0$ and a gradient soliton if $X=\nabla f$ for some function $f$. We always write a triple $(M,g,f)$ to denote a gradient Ricci soliton and we say a Rimannian manifold $(M,g)$ has a gracient soliton structure if there is a function $f$ satisfies the soliton equation.

We consider the isoperimetric problem on manifolds that have a steady gradient Ricci soliton structure. Because of the Ricci soliton equation, we find that the isoperimetric inequality still holds in the cigar steady soiton and in the Bryant steady soliton. Our main theorems are the following.

\begin{thm}\label{2soliton}
  Let $(\mathbb{R}^2,g,f)$ be the cigar steady Ricci soliton. If $\gamma\subset\mathbb{R}^2$ is a graph over $\mathbb{S}^1$, the length of $\gamma$ and the area of the compact domain whose boundary is $\gamma$ are denoted by $L(\gamma)$ and $A(\gamma)$. Then
  \[ L(\gamma)\geq F(A(\gamma)) \]
  with equality holds if and only if $\gamma$ is a circle $\{r\}\times \mathbb{S}^1$, where $F$ is a single variable function which connects the relation between the length and the area of circles in the cigar (see Theorem \ref{2ineq} in section 2).
\end{thm}

\begin{thm}\label{3soliton}
  Let $(M,g,f)$ be the complete steady Bryant soliton with $M=(0,\infty)\times\mathbb{S}^{n-1}$, $g=dr^2+\phi(r)^2g_{\mathbb{S}^{n-1}}$. Let $\Sigma\subset M$ be a hypersurface which is defined as $r=\rho(p)$, $p\in\mathbb{S}^{n-1}$ for a smooth function $\rho$ on $\mathbb{S}^{n-1}$. Let $S(r_0)$ be the level set of $r$ and $\Omega$ be the domain bounded by $\Sigma$ and $S(r_0)$. Then
  \[ Area(\Sigma)\geq\xi(Vol(\Omega)) \]
  where $\xi$ is a well-defined single-valued function that relates the area and volume of spheres in $(M,g)$ (see Theorem \ref{3ineq} in section 2). Moreover, the equality is attained if and only if $\Sigma$ is a level set of $r$.
\end{thm}

\hspace{0.4cm}

\noindent {\bf Acknowledgements}. I would like to express my gratitude to my advisor Professor Jiayu Li. He introduced me the problem and gave me so much useful suggestions.

\section{The isoperimetric inequality in Riemannian manifolds with warped product metric}

In the paper of Guan-Li-Wang \cite{Guan2019}, they proved an isoperimetric inequality by investigating a deformed mean curvature type flow which preserves the volume but decreases the area in manifolds with warped product metric. We will briefly state their results in this section.

Let $(\mathbb{B}^{n-1},\tilde{g})$ be a closed Riemannian manifold and $\phi=\phi(r)$ be a smooth positive function defined on the interval $[r_0,r_1]$ for some $r_0<r_1$. Consider a Riemannian manifold of warped product metric $(\mathbb{N}^n,\bar{g})$,
\begin{equation}\label{1}
  \bar{g}=dr^2+\phi^2\tilde{g},\ \ r\in[r_0,r_1]
\end{equation}
$X=\phi(r)\partial_r$ is a conformal Killing field of $\mathbb{N}^n$, i.e. $\mathcal{L}_X\bar{g}=2\phi'(r)\bar{g}$. Let $M$ be a smooth closed embedded hypersurface in $\mathbb{N}^n$ with an embedding $F_0$. Consider the following evolution equation for a family of embeddings of hypersurfaces with $F_0$ as an initial data,
\begin{equation}\label{2}
  \frac{\partial F}{\partial t}=((n-1)\phi'-uH)\nu
\end{equation}
where $\nu$ is the outward unit normal vector field, $H$ is the mean curvature, $u=<X,\nu>$. A hypersurface is said to be graphical if it is defined by $r=\rho(p)$, $p\in\mathbb{B}^{n-1}$ for a smooth function $\rho$ on $\mathbb{B}^{n-1}$.
In \cite{Guan2019}, they proved
\begin{thm}\label{3flow}
  Let $M_0$ be a smooth graphical hypersurface in $(\mathbb{N}^n,\bar{g})$ with $n\geq 3$ and $\bar{g}$ in \eqref{1}. If $\phi(r)$ and $\bar{g}$ satisfy the following conditions:
  \begin{equation}\label{3}
  \begin{aligned}
    &\tilde{Ric}\geq (n-2)K\tilde{g},\\
    &0\leq (\phi')^2-\phi''\phi\leq K\ \  on [r_0,r_1]
    \end{aligned}
  \end{equation}
  where $K>0$ is a constant and $\tilde{Ric}$ is the Ricci curvature of $\tilde{g}$. Then the evolution equation \eqref{2} with $M_0$ as the initial data has a smooth solution for $t\in[0,\infty)$. Moreover, the solution hypersurface converge exponentially to a level set of $r$ as $t\rightarrow\infty$.
\end{thm}

From the long-time existence of the flow \eqref{2}, they get an isoperimetric inequality for warped product space. Let $S(r)$ be a level set of $r$ and $B(r)$ be the bounded domain enclosed by $S(r)$ and $S(r_0)$. The volume of $B(r)$ and surface area of $S(r)$ are denoted by $V(r)$ and $A(r)$, respectively. There is a well-defined single variable function $\xi(x)$ that satisfies
\begin{equation}\label{4}
  A(r)=\xi(V(r))
\end{equation}
for any $r\in [r_0,r_1]$.

\begin{thm}\label{3ineq}
  Let $\Omega\subset\mathbb{N}^n$ be a domain bounded by a smooth graphical hypersurface $M$ and $S(r_0)$. Assume \eqref{3}, then
  \begin{equation}\label{5}
    Area(M)\geq \xi(Vol(\Omega))
  \end{equation}
  where $Area(M)$ is the area of $M$ and $Vol(\Omega)$ is the volume of $\Omega$. If either $(\phi')^2-\phi''\phi<K$ or $\tilde{Ric}>(n-2)K\tilde{g}$ on $[r_0,r_1]$ then $"="$ is attained in \eqref{5} if and only if $M$ is a level set of $r$.
\end{thm}

For $n=2$, similar result is proved by Dylan Cant in \cite{Cant2016A}.

Let $N=(0,R)\times\mathbb{S}^1$ be a surface with metric $g=dr^2+\phi(r)^2d\theta^2$, $\phi(r)>0$.

\begin{thm}\label{2flow}
  Let $N^2$ be a warped product space with warp potential $\phi(r)$ satisfying $(\phi')^2-\phi\phi''\geq0$. If $\gamma_0\subset N$ is a smooth hypersurface, then there is a unique flow $\gamma(t)$ with \eqref{2} and $\gamma(0)=\gamma_0$.
\end{thm}

Let $C_r$ denote the circle $\{r\}\times\mathbb{S}^1\subset N$, and $L(r)$ and $A(r)$ denote its length and area. There is some function $F$ with $L(r)=F(A(r))$.

\begin{thm}\label{2ineq}
  If $\gamma_0\subset N$ is a piecewise $C^1$ Lipschitz radial graph and $(\phi')^2-\phi\phi''\in[0,1]$, then $L(\gamma_0)\geq F(A(\gamma_0))$, where equality holds if and only if $\gamma_0$ is a circle $\{r\}\times\mathbb{S}^1$ if $(\phi')^2-\phi\phi''\not\equiv 1$.
\end{thm}

\section{Cigar steady soliton}

We wonder if the isoperimetric inequality is still available on Ricci solitons. First, we consider the steady solitons of dimension $n=2$.  We have the following Theorem in \cite{ChowThe} which implies that the complete steady gradient Ricci soliton with positive curvature is the cigar soliton.

\begin{thm}
  If $(M^2,g(t))$ is a complete steady gradient Ricci soliton with positive curvature, then $(M^2,g(t))$ is the cigar soliton.
\end{thm}

Let $(\mathbb{R}^2,g)$ be a complete Riemannian manifold with complete metric $g=\frac{dx^2+dy^2}{1+x^2+y^2}$. In polar coordinates, we can write \[g=\frac{dr^2+r^2d\theta^2}{1+r^2}.\]
Let $s=\text{arcsinh} r=\log(r+\sqrt{1+r^2})$, then
\[ g=ds^2+\tanh^2sd\theta^2. \]
By the gradient steady Ricci soliton equation
\[ Ric(g)+\nabla^2f=0 \]
we have that $(\mathbb{R}^2,g)$ is a gradient steady Ricci soliton with potential function
\[ f(s)=-2\log(\cosh s) \]
and its curvature is
\[ Ric(g)=\frac{2}{\cosh^2s}g. \]
This soliton is said to be the Cigar steady soliton.

Corresponding to Theorem \ref{2flow} and Theorem \ref{2ineq}, in the Cigar steady soliton case,
\[\phi(s)=\tanh s\]
we can calculate directly to get
\[ \phi'(s)=\frac{1}{\cosh^2 s}, \ \  \phi''(s)=-2\frac{\sinh s}{\cosh^3 s}.  \]
Thus
\[ (\phi')^2-\phi\phi''=\cosh^{-4}s(1+2\sinh^2 s)>0  \]
and
\[ (\phi')^2-\phi\phi''=\frac{2\cosh^2s-1}{\cosh^4 s}=1-\frac{(\cosh^2-1)^2}{\cosh^4s}\leq 1. \]
Therefore, Theorem \ref{2ineq} implies that the isoperimetric inequality is still true for cigar steady soliton. Theorem \ref{2soliton} is concluded.

\section{Bryant steady soliton}
In this section, we focus on the gradient steady Ricci solitons which is radial symmetric with $n\geq 3$.
Let $g_{\mathbb{S}^{n-1}}$ be the standard metric on the unit $(n-1)$-sphere. We will search for gradient steady Ricci solitons on $(0,\infty)\times \mathbb{S}^{n-1}$ which extend to Ricci solitons on $\mathbb{R}^n$ by a one point compactification of one end.
Consider the metric
\[ g=dr^2+\phi^2(r)g_{\mathbb{S}^{n-1}} \]
its Ricci curvature is
\[ Ric(g)=-(n-1)\frac{\phi''}{\phi}dr^2+((n-2)(1-(\phi')^2)-\phi\phi'')g_{\mathbb{S}^{n-1}}. \]
The Hessian of a function $f$ with respect to $g$ is
\[ \nabla^2f(r)=f''(r)dr^2+\phi\phi'f'g_{\mathbb{S}^{n-1}}. \]
From the steady Ricci soliton equation
\[ Ric(g)+\nabla^2f=0 \]
we get the following ODE system
\[ f''=(n-1)\frac{\phi''}{\phi} \]
\[ \phi\phi'f'=-(n-2)(1-(\phi')^2)+\phi\phi''. \]
Making the transformations
\[ x=\phi',\ \ y=(n-1)\phi'-\phi f',\ \ dt=\frac{dr}{\phi} \]
the ODE system becomes
\begin{equation}\label{ode}
  \begin{aligned}
     & \frac{dx}{dt}=x(x-y)+(n-2), \\
     & \frac{dy}{dt}=x(y-(n-1)x)
  \end{aligned}
\end{equation}
We will check the condition \eqref{3}. Notice that the first assumption of \eqref{3} is naturally satisfied because $Ric(g_{\mathbb{S}^{n-1}})=(n-2)g_{\mathbb{S}^{n-1}}$ and $K=1$. We calculate
\begin{equation}\label{con3}
  \begin{split}
     (\phi')^2-\phi\phi''= & (\phi')^2-\phi\phi'f'-(n-2)(1-(\phi')^2) \\
       =& -\phi\phi'f'+(n-1)(\phi')^2-(n-2) \\
       =& xy-(n-1)x^2+(n-1)x^2-(n-2) \\
       =& xy-(n-2).
  \end{split}
\end{equation}
Thus, the second condition of \eqref{3} is equivalent to
\begin{equation}\label{con33}
  n-2\leq xy\leq n-1.
\end{equation}
We can see the ODE system has constant solutions $(1,n-1)$ and $(-1,-n+1)$. They both satisfy \eqref{con33}.

\begin{lemma} \cite{ChowThe}
  Let $0<L\leq\infty$ and let $g$ be a warped product metric on the topological cylinder $(0,L)\times\mathbb{S}^{n-1}$ of the form
  \[ g=dr^2+w(r)^2g_{\mathbb{S}^{n-1}}. \]
  Then $g$ extends to a smooth metric as $r\rightarrow0_+$ if and only if
  \[ \lim_{r\rightarrow0_+}w(r)=0, \]
  \[ \lim_{r\rightarrow0_+}w'(r)=1, \]
  \[ \lim_{r\rightarrow0_+}\frac{d^{2k}w}{dr^{2k}}(r)=0 \]
  for all $k\in \mathbb{N}$.
\end{lemma}

Because of the above lemma, we need the metric to close up smoothly as $r\rightarrow0$, so
$x\rightarrow1$, $y\rightarrow n-1$. And $x\rightarrow0_+$, $y\rightarrow+\infty$ as $t$ increases. From \eqref{ode}, at the saddle point $(1,n-1)$, we can see that $x$ is decreasing and $y$ is increasing. Let
\begin{equation*}
  X=\sqrt{n-1}\frac{x}{y},\ \ Y=\frac{\sqrt{(n-1)(n-2)}}{y}, \ \ ds=ydt
\end{equation*}
then as $x\rightarrow0_+$, $y\rightarrow+\infty$, we see $X\rightarrow0_+$ decreasingly and $Y\rightarrow0_+$ decreasingly too.
The ODE system turns out to be
\begin{equation*}
  \begin{aligned}
  \frac{dX}{ds}=X^3-X+\alpha Y^2,\\
  \frac{dY}{ds}=Y(X^2-\alpha X)
  \end{aligned}
\end{equation*}
where $\alpha=\frac{1}{\sqrt{n-1}}$. We have the following lemma from \cite{ChowThe}.

\begin{lemma}
  We have
  \[ \lim_{s\rightarrow+\infty}\frac{X}{Y^2}=\alpha \]
  where $\alpha=\frac{1}{\sqrt{n-1}}$.
\end{lemma}

Then,
\begin{equation*}
 \lim_{s\rightarrow+\infty} xy=(n-2)\sqrt{n-1}\lim_{s\rightarrow+\infty}\frac{X}{Y^2}=n-2.
\end{equation*}

\begin{lemma}
 We have
 \[ \frac{X}{Y^2}<\frac{1}{(n-2)\alpha}\]
 for $s>0$, where $\alpha=\frac{1}{\sqrt{n-1}}$. Therefore,
 \[ xy=\frac{X}{Y^2}(n-2)\sqrt{n-1}< n-1. \]
\end{lemma}
\begin{proof}
  From the ODE system, we can calculate directly
  \begin{equation*}
    \begin{split}
       \frac{d}{ds}\left(\frac{X}{Y^2}\right)= & \frac{X^3-X+\alpha Y^2}{Y^2}-\frac{2X(X^2-\alpha X)}{Y^2} \\
        = & \alpha-\frac{X}{Y^2}(X^2-2\alpha X+1).
    \end{split}
  \end{equation*}
  Since as $s\rightarrow0$, $xy\rightarrow n-1$, we can see $\frac{X}{Y^2}\rightarrow \frac{1}{(n-2)\alpha}$. If $s_0>0$ is the first point such that, at $s_0$,
  \[ \frac{X}{Y^2}\geq\frac{1}{(n-2)\alpha} \]
  then, at $s_0$,
  \begin{equation*}
    \begin{split}
       \frac{d}{ds}\left(\frac{X}{Y^2}\right)= & \alpha-\frac{X}{Y^2}((X-\alpha)^2+1-\alpha^2) \\
        \leq & \alpha-\frac{1}{(n-2)\alpha}((X-\alpha)^2+1-\alpha^2)\\
        = & -\frac{1}{(n-2)\alpha}(X-\alpha)^2\\
        < & 0
    \end{split}
  \end{equation*}
  where the last step follows from the fact that $X$ is decreasing and so $X<\alpha$ for $s_0>0$.
  This tells us that $\frac{X}{Y^2}$ is decreasing at $s_0$. Then there is an $\epsilon>0$ small, such that $s_0-\epsilon>0$ and
  \[ \frac{X}{Y^2}(s_0-\epsilon)>\frac{X}{Y^2}(s_0)\geq\frac{1}{(n-2)\alpha} \]
  which contradicts to the fact that $s_0>0$ is the first point satisfying $\frac{X}{Y^2}\geq\frac{1}{(n-2)\alpha}$.
\end{proof}
Notice that the Bryant steady soliton has positive curvature away from the origin and the sectional curvature of the plane tangent to the radial direction is $-\frac{\phi''}{\phi}$, then we obtain that $\phi''<0$ and thus $(\phi')^2-\phi\phi''>0$. Therefore, the Bryant steady soliton satisfies the condition \eqref{3} of Theorem \ref{3flow}, which implies the isoperimetric inequality \eqref{5} by Theorem \ref{3ineq}. Theorem \ref{3soliton} is proved.

\begin{remark}
Ievy also constructed steady solitons on doubly warped product metric \cite{Ivey1994}. Let $(M^n,d\sigma^2)$ be a compact Einstein manifold with Einstein constant $\epsilon>0$, and $d\theta^2$ be the standard metric on $\mathbb{S}^{k}$. Consider the metric
\[ g=dr^2+f(r)^2d\theta^2+h(r)^2d\sigma^2 \]
on $\mathbb{R}\times\mathbb{S}^k\times M^n$. We do not know whether the isoperimetric inequality holds in this steady soliton and this is an open question.
\end{remark}
\bibliography{Iso-ineq-in-Ricci-solitons}

\end{document}